\documentclass{amsart}
\usepackage{url}
\usepackage[dvips]{graphicx}

\newtheorem{theorem}{Theorem}[section]
\newtheorem{lemma}[theorem]{Lemma}
\newtheorem{corollary}[theorem]{Corollary}

\newtheorem{proposition}[theorem]{Proposition}
\theoremstyle{definition}  
\newtheorem{definition} [theorem] {Definition} 

\newtheorem{example} [theorem] {Example}
\newtheorem{remark} [theorem] {Remark}

\newcommand{\C}{{\mathbb{C}}}

\newcommand{\Q}{{\mathbb{Q}}}

\newcommand{\Z}{{\mathbb{Z}}}

\newcommand{\Zi}{{\mathbb{Z}[i]}}
\newcommand{\LL}{{\Lambda}}

\newcommand{\OO}{{\mathcal{O}}}

\newcommand{\fp}{{\mathfrak{P}}}

\newcommand{\vf}{\varphi}
\newcommand{\vp}{\varpi}
\newcommand{\ve}{\varepsilon}

\newcommand{\gal}{\mathrm{Gal}}

\newcommand{\db}{\delta_\beta}

\newcommand\res[1]{{\lower1pt\hbox{$|$}}_{\raise.5pt\hbox{${\scriptstyle #1}$}}}

\begin{document}

\title{The Galois Theory of the Lemniscate}

\author{David A.\ Cox}
\address{Department of Mathematics, Amherst
College, Amherst, MA 01002-5000, USA}
\email{dac@math.amherst.edu}

\author{Trevor Hyde}
\address{Department of Mathematics, Amherst
College, Amherst, MA 01002-5000, USA}
\email{thyde641@gmail.com}

\begin{abstract}
This article studies the Galois groups that arise from division points
of the lemniscate.  We compute these Galois groups two ways: first, by
class field theory, and second, by proving the irreducibility of
\emph{lemnatomic polynomials}, which are analogs of cyclotomic
polynomials.  We also discuss Abel's theorem on the lemniscate and
explain how lemnatomic polynomials relate to Chebyshev polynomials.
\end{abstract}


\maketitle

\section{Introduction}

The lemniscate is the curve defined by the polar equation $r^2 =
\cos(2\theta)$: 
\[
\includegraphics[height=1.5in]{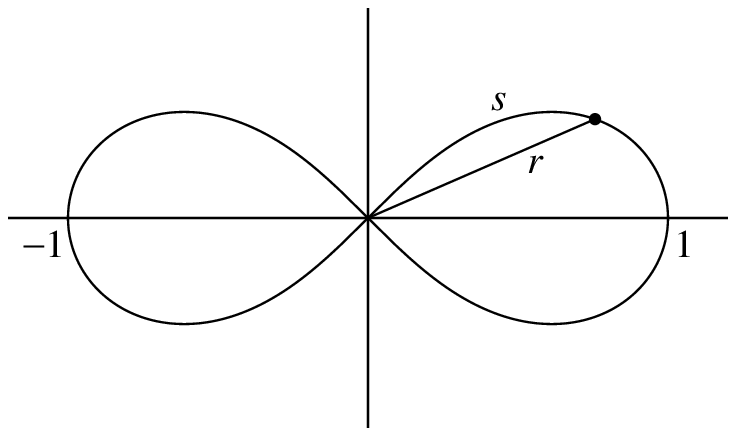}
\]
In the first quadrant, the arc length $s$ is related to the radial
distance $r$ by the elliptic integral
\begin{equation}
\label{arclength}
s = \int_0^r \frac{dt}{\sqrt{1-t^4}}.
\end{equation}
The arc length of the first-quadrant portion is $\int_0^1
\frac{dt}{\sqrt{1-t^4}}$, which we denote by $\varpi/2$ in analogy
with $\pi/2 = \int_0^1 \frac{dt}{\sqrt{1-t^2}}$.  Thus $2\varpi$ is
the length of the entire lemniscate.

Following Abel, the inverse function of \eqref{arclength} is denoted
$\varphi(s) = r$.  The goal of this paper is to compute the Galois
group
\begin{equation}
\label{galgrp}
\mathrm{Gal}(K(\varphi({\textstyle\frac{2\varpi}{n}}))/K),
\end{equation}
where $K = \Q(i)$ and $n$ is a positive odd integer.  Geometrically,
$\varphi(\frac{2\varpi}{n})$ tells us how to find the first
$n$-division point of the lemniscate, and the size of the Galois group
\eqref{galgrp} determines whether or not we can divide the lemniscate
into $n$ pieces by ruler and compass.  This will lead to a quick proof
of Abel's theorem on the lemniscate, which characterizes those $n$'s
for which this can be done (see Section~\ref{AbelThmSec} for a precise
statement).

Our main result is an isomorphism
\begin{equation}
\label{mainison}
\mathrm{Gal}(K(\varphi({\textstyle\frac{2\varpi}{n}}))/K) \simeq
\big(\mathbb{Z}[i]/n\mathbb{Z}[i]\big)^\times 
\end{equation}
when $n > 0$ is odd.  We will give two proofs:
\begin{itemize}
\item The first proof uses class field theory and complex
  multiplication.  The key step is to show that
  $K(\varphi(\frac{2\varpi}{n}))$ is the ray class field of $K =
  \mathbb{Q}(i)$ for the modulus $2(1+i)n$.
\item The second proof is more elementary and uses \emph{lemnatomic
  polynomials}, which are analogs of cyclotomic polynomials.  The key
  step is to prove that lemnatomic polynomials are irreducible.
\end{itemize}

Here is a brief summary of the paper.  Section~\ref{LemFnSec} explains
how $\varphi(s)$ extends to an elliptic function $\varphi(z)$ and
describes the period lattice $L$ and associated elliptic curve $E$.  We
also recall how complex multiplication by $\mathbb{Z}[i]$ gives
explicit formulas for $\varphi(\beta z)$ when $\beta \in
\mathbb{Z}[i]$.  Section~\ref{PrelimSec} introduces the field
$K_\beta = K(\vf(\frac{2\vp}{\beta}))$ for $\beta \in \Zi$
relatively prime to $1+i$ (we say that $\beta$ is \emph{odd}) and
constructs an injection
\begin{equation}
\label{Kbetainj}
\mathrm{Gal}(K_\beta/K)\hookrightarrow (\Z[i]/\beta\Z[i])^\times.
\end{equation}
We also give some alternate descriptions of $K_\beta$ that clarify its
relation to the elliptic curve $E$.  In Section~\ref{CFTSec} we
identify $K_\beta$ as the ray class field of $K$ for the modulus
$2(1+i)\beta$ when $\beta \in \Z[i]$ is odd and use class field theory
and complex multiplication to prove that \eqref{Kbetainj} is an
isomorphism.  This gives our first proof of \eqref{mainison}.
Section~\ref{LemPolySec} defines the lemnatomic polynomial
$\Lambda_\beta$ and proves its irreducibility over $K$ by an elementary
argument, leading to our second proof of \eqref{mainison}.  We also
determine the degree and constant term of $\Lambda_\beta$.
Section~\ref{AbelThmSec} shows how Abel's theorem on the lemniscate
follows from the irreducibility of lemnatomic polynomials, and the
final Section~\ref{ChebySec} explores a surprisingly strong analogy
between lemnatomic polynomials and irreducible factors of Chebyshev
polynomials.

\section{The Complex Lemniscatic Function}
\label{LemFnSec}

As explained in \cite[Ch.\ 15]{galois}, the function $\varphi(s)$ from 
the Introduction extends to a function of period $2\varpi$ on
$\mathbb{R}$ and satisfies the addition law
\begin{equation}
\label{addlaw}
\varphi(x+y) = \frac{\varphi(x)\varphi'(y) +
  \varphi(y)\varphi'(x)}{1+\varphi(x)^2\varphi(y)^2} 
\end{equation}
and the differential equation
\begin{equation}
\label{diffeq}
\varphi'(s)^2 = 1 - \varphi(s)^4. 
\end{equation}

To extend $\varphi$ to $\mathbb{C}$, note that the integral
\eqref{arclength} suggests defining $\varphi(iy) = i\varphi(y)$.  This
and \eqref{addlaw} enable us to define
\[
\varphi(z) = \varphi(x+iy) = \frac{\varphi(x)\varphi'(y) + i
  \varphi(y)\varphi'(x)}{1-\varphi(x)^2\varphi(y)^2}. 
\]
By \cite[Sec.\ 15.3]{galois}, $\vf$ is an elliptic function for the
lattice 
\begin{equation}
\label{Ldef}
L = \Z(1+i)\vp+\Z(1-i)\vp
\end{equation}
whose zeros (all simple) are at
points $\equiv 0,\vp \bmod L$ and whose poles (also simple) are at
points $\equiv (1\pm i)\vp/2 \bmod L$.  Furthermore, $\vf$ is an odd
function and 
\begin{equation}
\label{vfzvfw}
\vf(z) = \vf(w) \iff z = (-1)^{m+n} w + (m+in)\vp,\ m,n
\in \Z.
\end{equation}

For the rest of the paper, we will use the notation
\[
K = \Q(i),\qquad
\OO = \Zi.
\] 
Earlier we defined $\beta \in \OO$ to be \emph{odd} if it is
relatively prime to $1+i$.  Note that
\begin{align*}
\beta \in \OO \text{ is odd} &\iff \beta = m+in,\ m,n \in \Z,\ m+n
\text{ odd}\\ &\iff \beta \equiv i^\varepsilon \bmod 2(1+i) \text{ for
  some } \varepsilon \in \{0,1,2,3\}.
\end{align*}
We say that $\beta \in \OO$ is  \emph{even} if it is not odd, i.e., if
$1+i$ divides $\beta$. 

If $\beta \in \OO$ is nonzero and $L$ is the period
lattice \eqref{Ldef}, then note that
\[
{\textstyle\frac1\beta} L /L \simeq \OO/\beta\OO
\]
as $\OO$-modules.  We say that $\delta \in \frac1\beta L$ \emph{is a
  $\beta$-torsion generator} if $[\delta] \in \frac1\beta L /L$
generates $\frac1\beta L /L$ as an $\OO$-module.  Any two
$\beta$-torsion generators $\delta,\delta'$ satisfy $\delta \equiv
\alpha\delta' \bmod L$ for some $[\alpha] \in
(\OO/\beta\OO)^\times$. We will use the $\beta$-torsion generator $\db
= \frac{(1+i)\vp}{\beta}$ frequently.

It is clear that $\vf$ has complex multiplication by $\OO$.  Here is a
precise description of what this means.

\begin{theorem}
\label{phiPropsThm}
Let $\beta\in \OO$ be odd. Then there exist relatively prime%coprime
polynomials $P_\beta(x), Q_\beta(x)\in \OO[x]$ and $\varepsilon \in
\{0,1,2,3\}$ such that
\begin{enumerate} 
\item For all $z\in\C$, $\vf(\beta z) = M_\beta(\vf(z))$, where
\[
M_\beta(x) = i^\varepsilon x\,\frac{P_\beta(x^4)}{Q_\beta(x^4)}.
\]
\item $\beta \equiv i^\varepsilon \bmod{2(1+i)}$.
\item $P_\beta(x)$ and $Q_\beta(x)$ have degree
  $(\mathrm{N}(\beta)-1)/4$, where $\mathrm{N}(\beta) =
  \beta\overline{\beta}$ is the norm of $\beta$. 
\item The \textbf{\boldmath{$\beta$}-division polynomial}
  $xP_\beta(x^4)$ has $\mathrm{N}(\beta)$ distinct roots given by 
  $\vf(\alpha\delta)$ for $[\alpha ] \in \OO/\beta \OO$ and $\delta\in
  \frac1\beta L$ a fixed $\beta$-torsion generator. 
\item $P_\beta(x)$ is monic and $Q_\beta(x) =
  x^{(\mathrm{N}(\beta)-1)/4}\,P_\beta(1/x)$.
\item Suppose $\pi$ is an odd prime in $\OO$ and let
  $d=(\mathrm{N}(\pi)-1)/4$. Then  
\[
P_\pi(x) = x^d + a_1x^{d-1} + \dots + a_d,
\]
such that each $a_j$ is divisible by $\pi$ and $a_d =
i^{-\varepsilon}\pi$. 
\end{enumerate}
\end{theorem}

\begin{proof}
Parts (1), (2), (3) and (5) are proved in \cite[Thm.\ 15.4.4]{galois},
and part (6) (due to Eisenstein) is proved in
\cite[Thm.\ 15.4.8]{galois}.  As for part (4), let $\db =
\frac{(1+i)\vp}{\beta}$.  For any $\alpha \in \OO$, $\vf(\alpha \db)$
is a root of $xP_\beta(x^4)$ since $\vf$ vanishes at $\beta\alpha\db
\in L$.  Now suppose that $\alpha, \alpha' \in \OO$ satisfy
$\vf(\alpha\db) = \vf(\alpha'\db)$.  By \eqref{vfzvfw}, we have
\[
\alpha\db = (-1)^{m+n}\alpha'\db + (m+in)\vp,\quad m,n \in \Z.
\]
This implies $\alpha(1+i) = (-1)^{m+n}\alpha'(1+i) + (m+in)\beta$, so
that $m+in$ is even.  It follows easily that $\alpha \equiv \alpha'
\bmod \beta$.  Thus $xP_\beta(x^4)$ has at least $\mathrm{N}(\beta) =
|\OO/\beta\OO|$ distinct roots, namely $\vf(\alpha \db)$ for $[\alpha]
\in \OO/\beta\OO$.  Since this polynomial has degree $\mathrm{N}(\beta)$, these
are all of its roots.  Then the same holds for any other
$\beta$-torsion generator $\delta$ since $\delta \equiv \alpha \db
\bmod L$ for some $[\alpha] \in (\OO/\beta\OO)^\times$.
\end{proof}

\begin{remark}
There are formulas for $\vf(\beta z)$ when $\beta$ is even, but they
involve both $\vf(z)$ and $\vf'(z)$.  An example is the duplication
law
\[
\vf(2z) = \frac{2\vf(z)\vf'(z)}{1+\vf(z)^4}.
\]
We also note that differentiating the formula for $\vf(\beta z)$ from
Theorem~\ref{phiPropsThm} leads to a formula for $\vf'(\beta z)$ as a
rational function of $\vf'(z)$.
\end{remark}

We conclude this section with a few words about the elliptic curve $E
= \C/L$ associated to the period lattice $L = \Z(1+i)\vp+\Z(1-i)\vp$
from \eqref{Ldef}.

\begin{lemma}
\label{Eweq}
The Weierstrass equation of $E$ is $Y^2 = 4X^3 + X$.
\end{lemma}

\begin{proof}
The Weierstrass equation of $E = \C/L$ is $Y^2 = 4X^3 -g_2(L)X
-g_3(L)$, and $g_3(L) = 0$ since $E$ has complex multiplication by
$\OO=\Zi$.  As for $g_2(L)$, we know from \cite{Rosen} that the
lattice $L' = \Z 2\vp+\Z 2\vp i$ has $g_2(L') = \frac14$.  Since $L' =
(1+i)L$, we obtain
\[
{\textstyle\frac14} = g_2(L') = g_2((1+i)L) = (1+i)^{-4}g_2(L) =
-{\textstyle\frac14}g_2(L). 
\]
Hence $g_2(L) = -1$, which gives the desired Weierstrass
equation. \end{proof} 

Let $\wp(z) = \wp(z;L)$ be the Weierstrass $\wp$-function for $L$.
The elliptic functions $\vf$ and $\vf'$ have period lattice $L$ and
hence are rational functions of $\wp$ and $\wp'$.  By analyzing the
behavior at zeros and poles, one can show that
\begin{equation}
\label{vfwp}
\vf(z) = -2\frac{\wp(z)}{\wp'(z)},\quad \vf'(z) =
\frac{4\wp(z)^2-1}{4\wp(z)^2+1}. 
\end{equation}
To see what this means geometrically, note that by \eqref{diffeq},
$\vf$ and $\vf'$ parametrize the curve 
\begin{equation}
\label{vfeq}
y^2 = 1-x^4,
\end{equation}
while $\wp$ and $\wp'$ parametrize the curve of Lemma~\ref{Eweq}.
Then \eqref{vfwp} tells us that the curves of Lemma~\ref{Eweq} and
\eqref{vfeq} are related by the birational transformation 
\[
x = -2\frac{X}{Y},\quad y = \frac{4X^2-1}{4X^2+1}
\] 
with inverse $X = \frac12 (1+y)/x^{2}$, $Y = -(1+y)/x^{3}$.

\section{Preliminary Analysis of the Galois Group}
\label{PrelimSec}

We begin with the field $K_\beta = K(\vf(\frac{2\vp}\beta))$ defined
in the Introduction.   

\begin{proposition}
\label{fieldKbeta}
If $\beta \in \OO$ is odd and $\delta$ is any $\beta$-torsion
generator, then 
\[
K_\beta = K(\vf(\delta)) = K(\vf(\delta),\vf'(\delta)) =
K(\wp(\delta),\wp'(\delta)) = K(E[\beta]), 
\]
where $K(E[\beta])$ is the field obtained from $K$ by adjoining the
coordinates of the $\beta$-torsion points of $E$.
\end{proposition}

\begin{proof}
We first show that $\frac{2\vp}{\beta}$ is a $\beta$-torsion
generator.  Since $\beta$ is odd, we have $u\beta + v(1-i) = 1$ for some
$u,v \in \OO$.  Multiplying this by $\db = \frac{(1+i)\vp}{\beta}$
gives $v\frac{2\vp}{\beta} \equiv \db\bmod L$, and our claim follows. 

To prove the first equality of the proposition, let $\delta,\delta'$
be $\beta$-torsion generators. As noted above, $\delta \equiv
\alpha\delta' \bmod L$ for some $[\alpha] \in (\OO/\beta\OO)^\times$,
where we may assume that $\alpha$ is odd since $\beta$ is (if $\alpha$
is even, replace it with $\alpha + \beta$.) Then $\vf(\delta) =
\vf(\alpha\delta')$ and the latter is in $K(\vf(\delta'))$ by part (1)
of Theorem~\ref{phiPropsThm}. This implies $K(\vf(\delta)) =
K(\vf(\delta'))$, and then the first equality follows using $\delta' =
\frac{2\vp}{\beta}$.

For the second equality, we use \eqref{addlaw} to obtain
\[
\vf({\textstyle\frac{2\vp}{\beta}}) =
\vf\Big({\textstyle\frac{(1-i)(1+i)\vp}{\beta}}\Big) = \vf((1-i)\db) =
\frac{(1-i)\vf(\db)\vf'(\db)}{1+\vf(\db)^4}. 
\]
This implies $\vf'(\db) \in K(\vf(\frac{2\vp}{\beta}),\vf(\db)) =
K(\vf(\db))$, which shows that the second equality holds for $\db$.
Then the addition laws for $\vf$ and $\vf'$ show that it also holds
for any $\beta$-torsion generator $\delta$. 

The third equality follows since \eqref{vfwp} is birational, and the
final equality follows since $(\wp(\delta),\wp'(\delta)) \in E[\beta]$
generates $E[\beta]$ as an $\OO$-module. 
\end{proof}

The final equality of Proposition~\ref{fieldKbeta} shows that
$K_\beta$ is a Galois extension of $K$ (this also follows from Theorem~\ref{phiPropsThm}).  Here is a preliminary result
about the Galois group of $K_\beta/K$. 

\begin{proposition}
\label{preliminary}
Let $\beta \in \OO$ be odd.  Then for any $\sigma \in
\gal(K_\beta/K)$, there is a unique $[\alpha] \in
(\OO/\beta\OO)^\times$ such that $\sigma(\vf(\delta)) =
\vf(\alpha\delta)$ for any $\beta$-torsion generator $\delta$.
Furthermore, the map $\sigma \mapsto [\alpha]$ defines an injective
homomorphism 
\[
\gal(K_\beta/K) \hookrightarrow (\OO/\beta\OO)^\times.
\]
\end{proposition}

\begin{proof} A similar result is proved in
\cite[Thm.\ 15.5.1]{galois}.  The proposition here follows by
essentially the same argument.  We omit the details.
\end{proof}

Our eventual goal is to prove that the map of
Proposition~\ref{preliminary} is an isomorphism.  

\section{Class Field Theory}
\label{CFTSec}

Using class field theory, we provide the first proof of our main
result. 

\begin{theorem}
\label{CFTthm}
Let $\beta\in \OO$ be odd. Then $K_\beta$ is the ray class field of
$K$ for the modulus $2(1+i)\beta$.  Furthermore, the map of
Proposition~{\rm\ref{preliminary}} is an isomorphism 
\[
\gal(K_\beta/K) \simeq (\OO/\beta \OO)^\times.
\]
\end{theorem}

\begin{proof}
Let $\beta\in\OO$ be odd. The Weierstrass equation of $E$ has $g_2 =
-1$ and $g_3 = 0$, so that the Weber function of $E$ is given by
$(1/g_2)\wp^2 = -\wp^2$.  Since $\delta_{2(1+i)\beta}$ is a
$2(1+i)\beta$-torsion generator, the theory of complex multiplication
(see \cite[Thm.\ 5.6]{Silverman}, for example) tells us that 
\[
L = K(\wp(\delta_{2(1+i)\beta})^2)
\]
is the ray class field of $K$ for the modulus $2(1+i)\beta$.
Furthermore, since $\OO$ is a PID, the Artin map induces an
isomorphism 
\[
\gal(L/K) \simeq (\OO/2(1+i)\beta\OO)^\times /\OO^\times.
\]
However, since $\beta$ is odd, we have
\begin{align*}
(\OO/2(1+i)\beta\OO)^\times /\OO^\times
&\simeq \big( (\OO/2(1+i)\OO)^\times \times (\OO/\beta\OO)^\times
\big)/\OO^\times\\ 
&\simeq (\OO/\beta\OO)^\times,
\end{align*}
where the last isomorphism follows because $\OO^\times = \{\pm1,\pm
i\}$ maps isomorphically onto $(\OO/2(1+i)\OO)^\times$.  It follows
that 
\[
\gal(L/K) \simeq (\OO/\beta\OO)^\times.
\]

Using the relation \eqref{vfwp} between $\vf'$ and $\wp^2$, we can
write the ray class field $L$ as 
\[
L = K(\vf'(\delta_{2(1+i)\beta})).
\]
We prove that $\vf'(\delta_{2(1+i)\beta}) \in K_\beta$ as follows.
First observe that $\delta_{2(1+i)\beta} = \frac{\vp}{2\beta}$. Now
consider the identity 
\begin{align}
\label{phiIdentity}
\vf'(z-\tfrac{\vp}{2}) = 2\frac{\vf(z)}{1+\vf(z)^2}.
\end{align}
Since $\vf(\frac{\vp}{2}) = 1$ and $\vf'(\frac{\vp}{2}) = 0$,
\eqref{phiIdentity} follows from the addition law for
$\vf(z-\frac{\vp}{2})$ by differentiation. Since $\beta$ is odd, there
exists a $\gamma\in\OO$ such that $2(1+i)\gamma - \beta = i^\ve$ for
$\ve\in\{0,1,2,3\}$.  Then $\gamma\db - \frac{\vp}{2} =
i^\ve\frac{\vp}{2\beta}$. Using $\vf'(iz) = \vf'(z)$ and
\eqref{phiIdentity}, we obtain
\[
\vf'(\tfrac{\vp}{2\beta}) = \vf'(i^\ve\tfrac{\vp}{2\beta}) =
\vf'(\gamma\db - \tfrac{\vp}{2}) =  
2\frac{\vf(\gamma\db)}{1+\vf(\gamma\db)^2}\in K(\vf(\db)) = K_\beta. 
\]
It follows that $K \subseteq L \subseteq K_\beta$.  Since these are
Galois extensions, we have 
\[
|\gal(K_\beta/K)| \ge |\gal(L/K)| = |(\OO/\beta\OO)^\times|.
\]
Combining this with the injection of Proposition~\ref{preliminary}, we
conclude that $L = K_\beta$ and that the injection of
Proposition~\ref{preliminary} is an isomorphism.  
\end{proof}

Setting $\beta = n$ in Theorem~\ref{CFTthm} gives the following
corollary.  

\begin{corollary}
If $n$ is a positive odd integer, then $K_n = K(\vf(\frac{2\vp}{n}))$
is a Galois extension of $K$ with Galois group  
\[
\gal(K_n/K) \simeq (\OO/n\OO)^\times.
\]
Furthermore, $K_n$ is the ray class field of $K$ for the modulus
$2(1+i)n$. 
\end{corollary}

This is the class field theory proof of the isomorphism
\eqref{mainison} from the Introduction.  In 1980, Rosen \cite{Rosen}
applied class field theory to the lemniscate.  He used the lattice $L'
= \Z 2\vp+\Z 2\vp i$, which corresponds to the elliptic curve $E' =
\C/L'$ defined by $Y^2 = 4X^3 - \frac14 X$.  In order to prove Abel's
theorem on the lemniscate (see Theorem~\ref{abel} below), he used the
isomorphism
\[
\gal(K(\wp({\textstyle\frac{2\vp}{n}};L')^2)/K) \simeq
(\OO/n\OO)^\times/\OO^\times 
\]
since $K(\wp(\frac{2\vp}{n};L')^2)$ is the ray class field of $K$ for
the modulus $n$.  In Rosen's approach, $n$ can be any positive
integer.  The first person to apply class field theory to the
lemniscate was Takagi in his 1903 thesis, where he showed that all
abelian extensions of $K = \Q(i)$ are generated by division values of
the lemniscatic elliptic function (see Schappacher \cite[p.\ 259]{Sch1}
for precise references).  

Theorem~\ref{CFTthm} for the case when $\beta$ is an odd prime power
was first proved by Lemmermeyer in \cite[Thm.\ 8.19]{Lemmermeyer}.

\section{Lemnatomic Polynomials}
\label{LemPolySec}

In this section we present the second proof of our main
result. Pursuing the analogy between the roots of unity and the
division points of the lemniscate leads to an algebraic theory for the
lemniscate akin to the circle's cyclotomy. We begin by developing the
foundations of the theory, introducing the \emph{lemnatomic
polynomials}, and finally proving their irreducibility over $K$. Our
main result \eqref{mainison} follows as a corollary.

Here is the key definition of this section.

\begin{definition}
\label{lempolydef}
Let $\beta\in \OO$ be odd and set $\db = \frac{(1+i)\vp}{\beta}$.  We
call 
\[
\LL_\beta(x)\  = \!\!\!\!\prod_{[\alpha]\in (\OO/\beta 
\OO)^\times} \!\!\!\! {(x-\vf(\alpha\db))}
\]
the \emph{\bfseries \boldmath{$\beta^{th}\!$} lemnatomic polynomial}.
\end{definition}

In this definition, we can replace $\db$ with any $\beta$-torsion
generator $\delta$.  Furthermore, the roots of $\LL_\beta$ are
$\vf(\delta)$ as $\delta$ ranges over all lattice-inequivalent
$\beta$-torsion generators.

Note also that $\LL_{\beta} = \LL_{\beta'}$ when $\beta$ and $\beta'$
are associates in $\OO$ since $\LL_\beta$ depends only on the ideal
generated by $\beta$. It will often be convenient to specify an
associate.  Recall that an odd $\beta \in \OO$ satisfies $\beta \equiv
i^\ve \bmod 2(1+i)$ for $\ve \in \{0,1,2,3\}$.

\begin{definition}
An odd element $\beta\in\OO$  is
\emph{\bfseries normalized} if $\beta\equiv 1\bmod{2(1+i)}$.
\end{definition}

\begin{lemma}
\label{features}
Let $\LL_\beta$ be as above. Then $\LL_\beta\in \OO[x]$ is monic
of degree $|(\OO/\beta \OO)^\times|$.
\end{lemma}

\begin{proof}
All that requires demonstration is $\LL_\beta\in \OO[x]$. By
Proposition \ref{preliminary}, the roots of $\LL_\beta$ are permuted
under the action of $\gal(K_\beta/K)$ and hence the coefficients lie
in the fixed field $K$ of the Galois group.  The roots of $\LL_\beta$
are algebraic integers since they are also roots of the monic
polynomial $xP_\beta(x^4)\in \OO[x]$.  It follows that
$\LL_\beta(x)\in \OO[x]$.
\end{proof}

As cyclotomic polynomials provide a factorization of $x^n-1$ over
$\Q$, the lemnatomic polynomials provide a factorization of
$xP_\beta(x^4)$ over $K$.

\begin{proposition}
\label{decomp}
Let $\beta\in \OO$ be odd and let $xP_\beta(x^4)$ be the
$\beta$-division polynomial from Theorem~{\rm\ref{phiPropsThm}}. Then
\[
xP_\beta(x^4) = \prod_{\gamma \mid \beta}{\LL_\gamma (x)},
\]
where the product is over all normalized divisors $\gamma$ of $\beta$.
\end{proposition}

\begin{proof}
From Theorem \ref{phiPropsThm}, the roots of $xP_\beta(x^4)$ are
$\vf(\alpha\db)$ for $[\alpha]\in\OO/\beta\OO$. The proof consists of
using gcd's to reorganize the roots as in the analogous cyclotomic
result. We refer the reader to \cite[Prop.\ 9.1.5]{galois}.
\end{proof}

We can also determine the constant term of $\LL_\beta$.

\begin{proposition}
\label{constantterm}
Let $\beta\in\OO$ be odd and a nonunit. If $\beta = u\pi^k$ where
$u\in\OO^\times$, $\pi$ is a normalized prime, and $k \ge 1$, then
$\LL_\beta(0) = \pi$. In all other cases, $\LL_\beta(0) = 1$.
\end{proposition}

\begin{proof}
We may assume that $\beta$ is odd and normalized (this implies $i^\ve
= 1$ in part (1) of Theorem~\ref{phiPropsThm}.) Using $\varphi'(0) = 
1$ and L'H\^{o}spital's rule, we have
\begin{equation}
\label{LHospital}
\beta = \lim_{x\rightarrow 0}{\frac{\vf(\beta x)}{\vf(x)}} =
\lim_{x\rightarrow 0}{\frac{P_\beta(\vf(x)^4)}{Q_\beta(\vf(x)^4)}} =
\frac{P_\beta(0)}{Q_\beta(0)} = P_\beta(0)
\end{equation}
since $Q_\beta(0) = 1$ by part (5) of Theorem~\ref{phiPropsThm}.

Now let $\pi$ be a normalized odd prime.  If $\beta = \pi$, then
Proposition \ref{decomp} tells us that $\LL_\pi(x) =
P_\pi(x^4)$. Therefore, $\LL_\pi(0) = P_\pi(0) = \pi$ by
\eqref{LHospital}.
	
Next, if $\beta = \pi^k$, then we proceed by strong induction. Suppose
that for all $1\leq k \leq n$, we have $\LL_{\pi^k}(0)=\pi$. From
Proposition \ref{decomp} and our induction hypothesis we obtain
\[
\pi^{n+1} = P_{\pi^{n+1}}(0) = \prod_{k=1}^{n+1}{\LL_{\pi^k}(0)} =
\pi^n\LL_{\pi^{n+1}}(0). 
\]
We conclude that $\LL_{\pi^{n+1}}(0)=\pi$, completing our induction.
	
Finally, suppose $\beta$ is not a prime power. Again, Proposition
\ref{decomp} tells us that
\begin{equation}
\label{eqn6}
\begin{aligned}
\beta = P_\beta(0) = \hskip-9pt\prod_{\substack{\gamma\mid\beta,\,
    \gamma\neq 1\\ \gamma
    \text{normalized}}}{\hskip-8pt\LL_\gamma(0)}.
\end{aligned}
\end{equation}
Each prime power $\pi^k$ dividing $\beta$, where $\pi$ is a normalized
prime, contributes $\pi$ to the product (\ref{eqn6}). Thus the
normalized prime
power divisors of $\beta$ contribute a factor of $\beta$ to
(\ref{eqn6}). Therefore
\[
\prod_{\substack{\gamma\mid\beta,\, \gamma\neq 1,\, \gamma \text{
      normalized}\\ \gamma \text{ not a  
      prime power}}}{\hskip-20pt\LL_\gamma(0)} = 1.  
\]
Using this and strong induction on the number of prime factors of
$\beta$ (counted with multiplicity), it is now easy to prove that
$\LL_\beta(0) = 1$ when $\beta$ is not a prime power.  We leave the
details to the reader.
\end{proof}

\begin{remark}
The inspiration for Proposition~\ref{constantterm} derives from a
similar result for irreducible factors of Chebyshev polynomials proved in
\cite[Prop.\ 2]{Hsiao}.  We will explore the analogy between
lemnatomic polynomials and Chebyshev polynomials in
Section~\ref{ChebySec}.
\end{remark}

We now come to the main result of this section.

\begin{theorem}
\label{main}
If $\beta\in \OO$ is odd, then the lemnatomic polynomial $\LL_\beta$
is irreducible over $K$.  
\end{theorem}

\begin{proof}
We will follow the strategy used in \cite[Thm.\ 9.1.9]{galois} to
prove that the cyclotomic polynomial $\Phi_n$ is irreducible over
$\Q$.
  
Since $\LL_\beta \in \OO[x]$ is monic, Gauss's lemma implies that an
irreducible factor $f$ of $\LL_\beta$ in $K[x]$ lies in $\OO[x]$, and
we can assume that $f$ is also monic.  The roots of $f$ have the form
$\vf(\alpha_i\db)$ for $\alpha_i \in \OO$ with
$\gcd(\alpha_i,\beta)=1$ and $1\leq i \leq r$ for some $r$. For 
$\pi\in\OO$ an odd prime, consider the polynomial
\[
f_\pi(x) = \prod_{i=1}^r{(x-\vf(\pi\alpha_i\db))}.
\]
Recall that $xP_\beta(x^4)$ is separable by
Theorem~\ref{phiPropsThm}.  It follows that $\LL_\beta$ is separable
over $K$ and hence is separable modulo any prime of $\OO$ not dividing
$\mathrm{disc}(\LL_\beta) \ne 0$. For our purposes, we assume that the
prime $\pi \in \OO$ satisfies
\begin{enumerate}
	\item[(A)] $\pi\equiv 1 \bmod{2(1+i)}$.
	\item[(B)] $\LL_\beta$ is separable modulo $\pi$.
	\item[(C)] $\gcd(\pi,\beta) = 1$.
\end{enumerate}

First, we show that $f_\pi\in \OO[x]$. Let $\sigma\in\gal(K_\beta/K)$
and take a root $\vf(\pi\alpha_i\db)$ of $f_\pi$.  Since $\alpha_i\db$ and
$\pi\alpha_i\db$ are $\beta$-torsion generators,
Proposition~\ref{preliminary} implies that there is $\gamma \in \OO$
relatively prime to $\beta$ such that
\begin{equation}
\label{sigmaroots}
\begin{aligned}
\sigma(\vf(\alpha_i\db)) &= \vf(\gamma\alpha_i\db)\\
\sigma(\vf(\pi\alpha_i\db)) &= \vf(\gamma\pi\alpha_i\db) =
\vf(\pi\gamma\alpha_i\db). 
\end{aligned}
\end{equation}
Since $\vf(\alpha_i\db)$ is a root of $f \in \OO[x]$, the first line of
\eqref{sigmaroots} implies that $\vf(\gamma\alpha_i\db)$ is a root as
well.  Then $\vf(\pi\gamma\alpha_i\db)$ is a root of $f_\pi$ by
definition.  But this root is $\sigma(\vf(\pi\alpha_i\db))$ by the
second line of \eqref{sigmaroots}.  It follows that $\sigma$ permutes
the roots of $f_\pi$, and then $f_\pi(x)\in \OO[x]$ follows easily.

We next prove the following claim: 
\begin{equation}
\label{claim}
f = f_\pi \text{ when $\pi$ satisfies (A), (B) and (C).}
\end{equation}
Suppose \eqref{claim} is false.  Since $f$ and $f_\pi$ are monic of
the same degree and $f$ is irreducible, $f \ne f_\pi$ implies that $f$
is coprime to $f_\pi$. The roots of $f_\pi$ are also roots of
$\LL_\beta$, so there exists a monic $h\in \OO[x]$ such that
$\LL_\beta = ff_\pi h$. Let us analyze the roots of $f_\pi
\bmod{\pi}$. {F}rom parts (5) and (6) of Theorem~\ref{phiPropsThm}, we
have
\begin{equation}
\label{candoit}
\begin{aligned}
P_\pi(x^4) &= x^{\mathrm{N}(\pi)-1} + a_1x^{\mathrm{N}(\pi)-5} + \dots +
a_{(\mathrm{N}(\pi)-1)/4}\\ 
Q_\pi(x^4) &= a_{(\mathrm{N}(\pi)-1)/4}x^{\mathrm{N}(\pi)-1} + \dots +
a_1x^4 + 1, 
\end{aligned}
\end{equation}
where $\pi$ divides all the coefficients $a_j$. Let $\OO_{K_\beta}$ be
the ring of integers of $K_\beta$ and let $\fp$ be
a prime ideal of $\OO_{K_\beta}$ dividing $\pi\OO_{K_\beta}$ and
suppose that $\vf(\alpha_i\db)$ is a root of $f$. Therefore
\[
\vf(\pi\alpha_i\db) =
\vf(\alpha_i\db)
\frac{P_\pi\big(\vf(\alpha_i\db)^4\big)}{Q_\pi\big(\vf(\alpha_i\db)^4\big)}  
\]
in $K_\beta$. By (\ref{candoit}), we may reduce this equality
modulo $\fp$ to get 
\begin{equation}
\label{artin}
\vf(\pi\alpha_i\db) \equiv \vf(\alpha_i\db)^{\mathrm{N}(\pi)} \bmod \fp.
\end{equation}
Let 
\[
\widetilde{f}_\pi(x) =
\prod_{i=1}^r{\big(x-\vf(\alpha_i\db)^{\mathrm{N}(\pi)}\big)}. 
\]
Then an argument similar to that given for $f_\pi$ shows that
$\widetilde{f}_\pi(x)\in\OO[x]$. By \eqref{artin}, $f_\pi(x)$ and
$\widetilde{f}_\pi(x)$ have the same roots in
$\OO_{K_\beta}/\fp$ and hence 
\[
f_\pi \equiv \widetilde{f}_\pi \bmod \pi \OO[x]. 
\]
However, $\widetilde{f}_\pi$ is obtained from $f$ by raising its roots
to the $\mathrm{N}(\pi)^{\mathrm{th}}$ power, and a standard argument
implies that
\[
 \widetilde{f}_\pi \equiv f \bmod \pi \OO[x].
 \]
(See \cite[Lem.\ 9.1.8]{galois} for the case of a prime $p \in \Z$.
The proof for a prime $\pi \in \OO$ is identical.)  Combining these
congruences shows that $f_\pi \equiv f \bmod \pi\OO[x]$.  Therefore,
\[
\LL_\beta \equiv ff_\pi h \equiv f^2h \bmod{\pi \OO[x]},
\]
which is to say that $\LL_\beta$ is not separable modulo $\pi$. This
contradicts our choice of $\pi$ and allows us to conclude
\eqref{claim}.

Now consider a root $\vf(\alpha_i\db)$ of $f$. Let $\eta$ be the product
of all odd primes dividing $\mathrm{disc}(\LL_\beta)$ but not dividing
$\beta$. If $\vf(\gamma\db)$ is any root of $\LL_\beta$, then
the Chinese remainder theorem implies that there is a Gaussian integer
$\omega$ such that
\[
\begin{array}{ccll}
\omega \hskip-5pt&\equiv &\hskip-5pt \gamma\alpha^{-1}_i
&\hskip-7pt\bmod{\,\beta}\\ 
\omega \hskip-5pt&\equiv &\hskip-5pt 1 &\hskip-7pt\bmod{\,2(1+i)}\\
\omega \hskip-5pt&\equiv &\hskip-5pt 1 &\hskip-7pt\bmod{\,\eta}.
\end{array}
\]
Therefore $\omega$ is odd and we may factor $\omega$ as $\omega =
\pi_1\cdots \pi_k$ for odd, normalized primes $\pi_j$ coprime to
$\beta\hskip1pt\mathrm{disc}(\LL_\beta)$. Iterating (\ref{claim}) we have
\[
f = f_{\pi_1} = f_{\pi_1\pi_2} = \cdots = f_{\pi_1\cdots\pi_k}.
\]
Hence
\[
\vf(\gamma\db) = \vf(\gamma\alpha^{-1}_i\alpha_i\db) =
\vf(\omega\alpha_i\db) = \vf(\pi_1\cdots\pi_k\alpha_i\db) 
\]
is a root of $f_{\pi_1\cdots\pi_k}=f$. Thus $\LL_\beta$ and $f$ have
the same roots; since both are monic and separable they must be
equal. We conclude that $\LL_\beta$ is irreducible over $K$ as
desired.
\end{proof}

\begin{corollary}
\label{maingalois}
Let $\beta\in \OO$ be odd. Then
\[
\gal(K_\beta/K)\simeq (\OO/\beta \OO)^\times.
\]
In particular, if $n$ is a positive, odd integer, then
\[
\gal(K(\vf({\textstyle\frac{2\vp}{n}}))/K) \simeq (\OO/n\OO)^\times.
\]
\end{corollary}

\begin{proof}
From Proposition \ref{preliminary}, we have an injection
\[
	\gal(K_\beta/K) \hookrightarrow (\OO/\beta \OO)^\times.
\]
Since $K_\beta = K(\vf(\db))$, Theorem \ref{main} and Lemma
\ref{features} tell us that  
\[
|\gal(K_\beta/K)|= [K_\beta:K] = [K(\vf(\db)):K] = \deg(\LL_\beta(x))
= |(\OO/\beta \OO)^\times|. 
\]
Therefore, our injection must be an isomorphism.  The final assertion
follows since $K_n = K(\vf(\frac{2\vp}{n}))$ by definition.
\end{proof}

The final assertion of Corollary~\ref{maingalois} completes the
elementary proof of the isomorphism \eqref{mainison} from the
Introduction. 

\begin{remark}
Let $\pi \in \OO$ be an odd prime not dividing $\beta$, so that $\pi$
is unramified in $K_\beta$.  Under the isomorphism of 
Corollary~\ref{maingalois}, some $\sigma \in \gal(K_\beta/K)$ maps to
$[\pi] \in (\OO/\beta\OO)^\times$, where $\sigma$ and $\pi$ are linked
via the equation
\[
\sigma(\vf(\db)) = \vf(\pi\db).
\]
If we assume in addition that $\mathrm{N}(\pi)$ is relatively prime to the
index $[\OO_{K_\beta}\! : \OO[\vf(\db)]]$, then \eqref{artin} easily
implies that 
\[
\sigma(u) \equiv u^{\mathrm{N}(\pi)} \bmod \fp
\]
for any $u \in \OO_{K_\beta}$ and any prime $\fp$ of
$\OO_{K_\beta}$ dividing $\pi\OO_{K_\beta}$.  Since $K_\beta/K$ is
abelian, the Artin symbol $((K_\beta/K)/\pi)$ is the unique element of
$\gal(K_\beta/K)$ with this property.  Hence we have proved that
\[
\left(\frac{K_\beta/K}{\pi}\right)(\vf(\db)) = \vf(\pi\db),
\]
which shows that the isomorphism $\gal(K_\beta/K) \simeq
(\OO/\beta\OO)^\times$ from Corollary~\ref{maingalois} can be
identified with the Artin map. 

Observe the key role played by \eqref{artin} here and in the proof of
Theorem~\ref{main}.  If you look back at the proof, you will see that
\eqref{artin} follows  from \eqref{candoit}.  Proving \eqref{candoit}
is actually the  hardest part of the proof of Theorem~\ref{main}.  The
proof of \eqref{candoit} given in \cite[Thm.\ 15.4.8]{galois} follows
Eisenstein's original argument, which is both intricate and
brilliant. 
\end{remark}

\begin{example}
Let us work out the case $n=5$.  The $5$-division polynomial is
\[
xP_5(x^4) = x^{25}+50x^{21}-125x^{17}+300x^{13}-105x^9-62x^5+5x.
\]
We can factor $5$ into normalized primes as $5 = (-1+2i)(-1-2i)$.  Then 
\begin{align*}
xP_5(x^4) &= \LL_1(x) \,\LL_{-1+2i}(x)\, \LL_{-1-2i}(x)\,\LL_5(x)\\
&= x(x^4-1+2i)(x^4-1-2i)(x^{16}+52x^{12}-26x^8-12x^4+1).
\end{align*}
Note that the constant terms of the factors are as predicted by
Proposition~\ref{constantterm}. 
\end{example}

\section{Abel's Theorem on the Lemniscate}
\label{AbelThmSec}

Our work in the previous section allows us to give a new, concise
proof of Abel's wonderful result about ruler and compass constructions
on the lemniscate. 

\begin{theorem}[Abel's Theorem on the Lemniscate]
\label{abel}
For a positive integer $n$, the $n$-division points of the lemniscate may be constructed using
ruler and compass if and only if 
\[
n=2^kp_1\cdots p_m,
\]
where the $p_i$ are distinct Fermat primes and $k$ is a nonnegative
integer. 
\end{theorem}

\begin{proof}
We begin by reducing the equivalence. According to
\cite[Prop.\ 15.1.1]{galois}, a point on the lemniscate is
constructible if and only if its radial component is
constructible. Therefore, constructing the $n$-division points is
equivalent to constructing $\vf(\frac{2\vp a}{n})$ for $a=1,\ldots,
n$. We need not worry about all these points, as
\cite[Cor.\ 15.2.7]{galois} states that if $\vf(\frac{2\vp}{n})$ is
constructible, then $\vf(\frac{2\vp a}{n})$ is constructible for all
$a\in\Z$. Finally, \cite[Prop.\ 15.2.3]{galois} allows us to conclude
that $\vf(z)$ is constructible if and only if $\vf(z/2)$ is
constructible. Hence we may assume $n$ is odd.

Next observe that $\LL_n \in \Z[x]$.  This follows since
lemnatomic polynomials satisfy $\overline{\LL_\beta} =
\LL_{\overline{\beta}}$ with respect to complex conjugation.  Then
$\LL_n$ is irreducible over $\Q$ since it is irreducible over $K$ by
Theorem~\ref{main}.  Furthermore, $\vf(\frac{2\vp}{n})$ is a root of
$\LL_n$ since $\frac{2\vp}{n}$ is an $n$-torsion generator (see the
proof of Proposition~\ref{fieldKbeta}).   It follows that  
\begin{equation}
\label{degreeformula}
[\Q(\vf({\textstyle\frac{2\vp}{n}})):\Q]\! = \!\deg(\LL_n)\! =
\!|(\OO/n\OO)^\times|\! =\! n^2\prod_{p\mid n}
\left(\!1-\frac1p\right)
\left(\!1-\left(\frac{-1}{p}\right)\frac1p\right), 
\end{equation}
where the last equality follows from \cite[Ex.\ 7.29]{P}.

Now assume that $\vf(\frac{2\vp}{n})$ is constructible.  Then
$[\Q(\vf(\frac{2\vp}{n})):\Q]$ is a power of $2$ by
\cite[Cor.\ 10.1.8]{galois}, so that by \eqref{degreeformula}, we have 
\[
2^m = n^2\prod_{p\mid n} \left(1-\frac1p\right)
\left(1-\left(\frac{-1}{p}\right)\frac1p\right). 
\]
Since $n$ is odd, this implies that $n$ is a product of distinct
primes $p$.  If $(-1/p) = -1$, then $(p-1)(p+1)$ is a power of $2$,
which forces $p = 3$, and if $(-1/p) = 1$, then $(p-1)^2$ is a power
of $2$, which easily implies that $p$ is a Fermat prime. 

Conversely, assume that $n$ is a product of distinct Fermat primes.
Observe that $\varphi(\frac{2\vp}{n}) \in K(\varphi(\frac{2\vp}{n})) =
K_n$, which is  Galois over $\Q$ since $K_n/K$ is Galois and
$\varphi(\frac{2\vp}{n})$ is real.  Then \eqref{degreeformula} shows
that  
\[
[K_n : \Q] = 2[\Q(\vf({\textstyle\frac{2\vp}{n}})):\Q] 
\]
is a power of $2$.  Hence $\vf(\frac{2\vp}{n})$ is constructible by
\cite[Thm.\ 10.1.12]{galois}. 
\end{proof}

Other proofs of this theorem are due to Abel, Eisenstein and Rosen,
though Rosen was the first to prove the ``only if'' part of the
theorem.  Rosen's proof \cite{Rosen} uses class field theory, while
the above proof uses only the irreducibility of the lemnatomic
polynomial $\LL_n$, which was proved without class field theory in
Theorem~\ref{main}.  The proofs of Abel and Eisenstein are discussed
in \cite[Sec.\ 15.5]{galois}, with references to the original papers.
See also the book \cite{pands} by Prasolov and Solovyev.

\section{Chebyshev Polynomials}
\label{ChebySec}

The analogy between Abel's function $\vf(z)$ and the sine function
$\sin(\theta)$ has been recognized since the time of Gauss and Abel.  For
example, in his unpublished work on elliptic functions, Gauss wrote
$\vf(z)$ as $\sin \hskip1pt \mathrm{lemn}\hskip1pt z$ (see
\cite{agm}). Thus the analogy between lemnatomic polynomials and
cyclotomic polynomials made earlier in this paper needs to be
reconsidered from the point of view of the sine function.  As we now
show, the analog of the division polynomial $xP_\beta(x^4)$ is the
well-known \emph{Chebyshev polynomial} $T_n \in \Z[x]$, which is
defined by the identity 
\[
	\cos(n\theta) = T_n(\cos(\theta))
\]
for a positive integer $n$.  

\begin{lemma}
\label{sinTn}
If $n$ is odd, then
\[
\sin(n\theta) = (-1)^{(n-1)/2} \,T_n\big(\sin(\theta)\big).
\]
Furthermore, there is $S_n \in \Z[x]$ such that
\[
T_n(x)  = x\,S_n(x^2),
\]
and the roots of $T_n(x) = x\,S_n(x^2)$ are $\sin(a\frac{2\pi}{n})$ for
$[a] \in \Z/n\Z$. 
\end{lemma}

\begin{proof} 
Since $n$ is odd, the addition laws for $\sin$ and $\cos$ imply that
\begin{align*}
\sin(n\theta) &= \cos\big(n\theta-{\textstyle\frac{\pi}{2}}\big) = 
\cos\big(n(\theta-{\textstyle\frac{\pi}{2}}) + 
{\textstyle\frac{n-1}{2}}\pi\big) =
\cos\big(n(\theta-{\textstyle\frac{\pi}{2}})\big)(-1)^{(n-1)/2}\\ &=
(-1)^{(n-1)/2}\,T_n\big(\cos(\theta-{\textstyle\frac{\pi}{2}})\big) =
(-1)^{(n-1)/2} \,T_n\big(\sin(\theta)\big). 
\end{align*}
The formula for $\sin(n\theta)$ implies that $0$ is a root of $T_n$.
Since $\sin(n\theta)/\sin(\theta)$ is an even function, it follows
that $T_n(x)/x$ is also even and hence is a polynomial in $x^2$.   The
assertion concerning the roots of $T_n$ is equally easy and is
omitted. 
\end{proof}

Notice how $\sin(n\theta) =
(-1)^{(n-1)/2}\,\sin(\theta)\,S_n(\sin(\theta)^2)$, $n$ odd, is
analogous to the formula for $\vf(\beta z)$, $\beta$ odd, from part
(1) of Theorem~\ref{phiPropsThm}, where $n \equiv (-1)^{(n-1)/2} \bmod
4$ corresponds to $\beta \equiv i^\ve \bmod 2(1+i)$.  Also, the
crucial identity \eqref{phiIdentity} used in the proof of
Theorem~\ref{CFTthm} is the lemniscatic analog of the identity
$\cos(\theta - \frac{\pi}2) = \sin(\theta)$ used in the proof of
Lemma~\ref{sinTn}.

The analogy between $\varphi(z)$ and $\sin(\theta)$ actually begins
with the equations $r^2 = \cos(2\theta)$ and $r = \cos(\theta)$ that
define the lemniscate and circle of radius $1/2$ centered at
$(1/2,0)$.  Table 1 shows how these curves lead naturally to the
functions $\varphi(z)$ and $\sin(\theta)$ and also records some of the
similarities between their properties.  (See also
\cite[8.2]{Lemmermeyer} for another view of the analogy between
$\varphi(z)$ and $\sin(\theta)$.) 

\begin{table}[t]
	\centering
	\begin{tabular}{|c|c|c|}
		\hline
		& Lemniscate\rule{0pt}{11pt} & Circle \\[2pt]
		\hline 
		Polar equation\rule{0pt}{11pt} & $r^2 = \cos(2\theta)$
                & $r = \cos(\theta)$\\[2pt] 
		\hline
		\raisebox{37pt}{Graph} &
                $\includegraphics[height=1in]{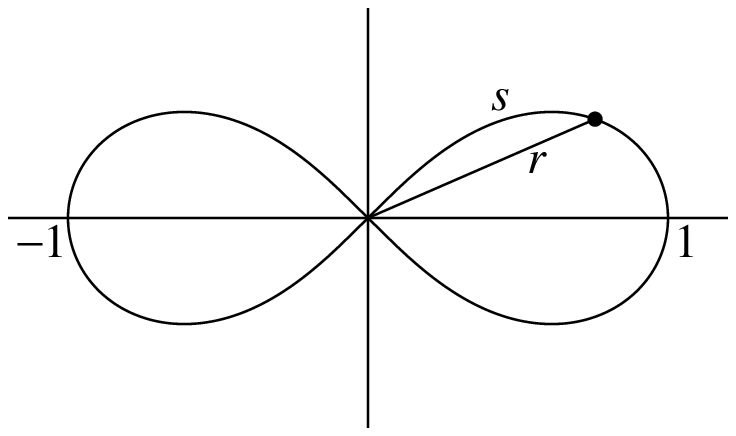}$ &
                $\includegraphics[height=1in]{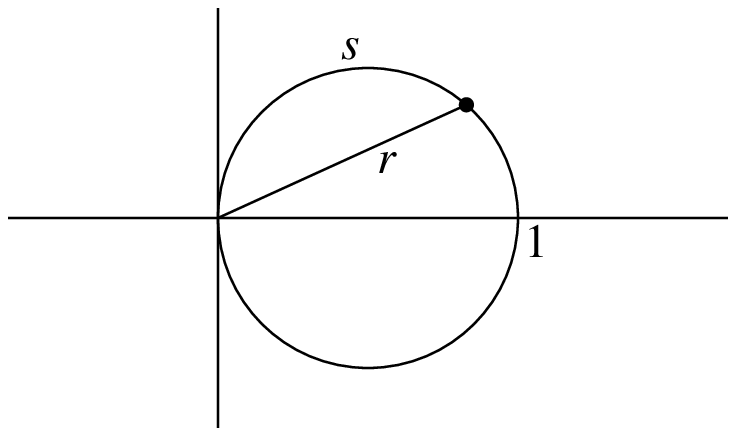}$\\ 
		\hline
		Arc length\rule{0pt}{17.5pt} & $s =
                {\displaystyle\int_0^r{\frac{dt}{\sqrt{1-t^4}}}}$ & $s =
                {\displaystyle\int_0^r{\frac{dt}{\sqrt{1-t^2}}}}$ \\[9pt] 
		\hline
		Inverse function\rule{0pt}{11pt} & $r = \vf(s)$ & $r =
                \sin(s)$ \\[2pt] 
	\hline
		\raisebox{-4pt}{Differential} &  &  \\
		\raisebox{6.5pt}{equation} & \raisebox{9pt}{$\vf'(z)^2
                  + \vf(z)^4 = 1$} & \raisebox{9pt}{$\cos(\theta)^2 +
                  \sin(\theta)^2 = 1$} \\[2pt] 
		\hline 
		\raisebox{-2pt}{Multiplication}\rule{0pt}{11pt} &
                $\vf(\beta z) =  $ & $\sin(n\theta) = $ \\ 
		\raisebox{3pt}{map ($\beta,n$ odd)} &
                $(-1)^{\ve}\vf(z)\,\frac{P_\beta(\vf(z)^4)}{Q_\beta(\vf(z)^4)}$
                & $(-1)^{(n-1)/2}\sin(\theta)\,S_n(\sin(\theta)^2)$
                \\[4pt] 
		\hline
		\raisebox{-1pt}{Galois group}\rule{0pt}{12pt} &
                $\gal(K(\vf(\tfrac{2\vp}{\beta}))/K)$ &
                $\gal(\Q(\sin(\tfrac{2\pi}{n}))/\Q)$ \\ 
		\raisebox{3pt}{($\beta,n$ odd)}\rule{0pt}{10pt} &
                $\simeq (\OO/\beta\OO)^\times$ & $\simeq
                (\Z/n\Z)^\times$ \\[2pt] 
		\hline
	\end{tabular}
	\medskip
\caption{The Lemniscate and the Circle}
\end{table}

This analogy suggests in particular that the lemnatomic polynomials
$\LL_\beta$ should correspond to the irreducible factors of the
Chebyshev polynomial $T_n$ when $n$ is odd.  The factors of $T_n$
were studied by Hsiao \cite{Hsiao} in 1984.  Since $T_n$ is not
monic, he used the monic polynomial 
\[
C_n(x) = 2\hskip1pt T_n(x/2) \in \Z[x].
\]
Hsiao determines the irreducible factorization of $C_n$ over $\Q$
\cite[Prop.\ 1]{Hsiao}.  When $n$ is odd, his result may be restated
as a factorization 
\begin{equation}
\label{factorCn}
C_n = \prod_{k|n} D_k,
\end{equation}
where $D_n$ has degree $\phi(n)$ (the Euler $\phi$-function) and is
given by 
\[
D_n(x) = \!\!\!\!\prod_{[a] \in (\Z/n\Z)^\times} \!\!\big(x -
2\sin(a{\textstyle\frac{2\pi}{n}})\big). 
\]
Note the analogy with Definition~\ref{lempolydef}.  Thus, when $n$ is
odd, \eqref{factorCn} is the analog of Proposition~\ref{decomp}, where
the irreducibility of $D_n$ corresponds to the irreducibility of
$\LL_\beta$ proved in Theorem~\ref{main}.  Although Hsiao's proof of
irreducibility ultimately rests on the the irreducibility of
cyclotomic polynomials, we note that the polynomials $D_n$ may be
shown irreducible over $\Q$ directly by adapting the proof of
Theorem~\ref{main}.  We also note that in one of his unpublished
papers, Schur \cite[p.\ 423]{Schur} mentions very briefly (without
proof) the irreducible decomposition of Chebyshev polynomials

Hsiao also studied the constant term of the factors of $C_n$.  When
$n$ is odd, his result can be stated as follows
\cite[Prop.\ 2]{Hsiao}.   

\begin{proposition}
\label{HsiaoConstant}
Let $n\in\Z$ be odd and positive. If $n=p^k$ for a prime $p$, then
$|D_n(0)| = p$. Otherwise, $D_n(0)=1$. 
\end{proposition}

By tweaking Hsiao's proof, one can show that when $p$ is an odd prime
and $k \ge 1$, we have $D_{p^k}(0) = (-1)^{(p-1)/2}\, p$.  Since
$(-1)^{(p-1)/2}\, p \equiv 1 \bmod 4$, we see that $(-1)^{(p-1)/2}\,
p$ is the analog of a normalized prime $\pi \in \OO$, so that
Proposition~\ref{HsiaoConstant} is the analog of
Proposition~\ref{constantterm}.  In fact,
Proposition~\ref{constantterm} was inspired by Hsiao's result.

%and Hsiao's result inspired us to
%prove Proposition~\ref{constantterm}.

For further results on the factorization of Chebyshev polynomials, see
\cite{RTW}.  There is also another interesting analogy to consider,
this one involving the \emph{Carlitz polynomials} $[M](X)$ for $M \in
\mathbb{F}_p[T][X]$.  The irreducible factors of Carlitz polynomials
have a lot in common with lemnatomic and cyclotomic polynomials.  See
\cite{conrad} for more details.  

For more on the history and number theory
associated to the lemniscate, the reader should consult Schappacher's
article \emph{Some Milestones of Lemniscatomy} \cite{Sch2}.

\section*{Acknowledegements}

The results of Section~\ref{LemPolySec} are based on the senior honors
thesis of the second author, written under the direction of the first
author.  We are grateful to Amherst College for the Post-Baccalaureate
Summer Research Fellowship that supported the writing of this paper.
We would like to thank Rosario Clement and Keith Conrad for useful
conversations.  The comments by Keith Conrad, Franz Lemmermeyer and
Michael Rosen on earlier versions of the paper are greatly
appreciated.

\bibliographystyle{model1b-num-names}

\hbadness=10000
\tolerance=10000

\end{document}